\documentclass[11pt, reqno]{amsart}
\numberwithin{equation}{section}
\usepackage{dsfont}
\usepackage{amscd,amsmath,amssymb,amsfonts}
\usepackage{tikz,color}
\usepackage{hyperref}
\usepackage{enumerate}
\usepackage{latexsym}
\usepackage{cases}
\usepackage{amsthm}
\usepackage{graphicx}
\usepackage{verbatim}
\usepackage{mathrsfs}
\usepackage{geometry}
\newcommand{\R}{\mathbb{R}}

\newtheorem{theorem}{Theorem}[section]

\newtheorem{lemma}[theorem]{Lemma}

\newtheorem{conjecture}[theorem]{Conjecture}
\newtheorem{proposition}[theorem]{Proposition}

\theoremstyle{definition}
\newtheorem{definition}[theorem]{Definition}
\newtheorem{remark}[theorem]{Remark}

\makeatletter
\newcommand{\Extend}[5]{\ext@arrow0099{\arrowfill@#1#2#3}{#4}{#5}}
\makeatother

\begin{document}
\title[Restriction estimate]{Restriction estimates for one class of hypersurfaces with vanishing curvature in $\mathbb{R}^n$}

\author[Z. Li]{Zhuoran Li}
	\address{School of Mathematics and Statistics, Yancheng Teachers University, Yancheng, 224002, China}
	\email{lizr@yctu.edu.cn}

\author[J. Zheng]{Jiqiang Zheng}
	\address{Institute of Applied Physics and Computational Mathematics, Beijing 100088}
	\address{National Key Laboratory of Computational Physics, Beijing 100088, China}
	\email{zhengjiqiang@gmail.com}

\begin{abstract}
In this paper, we study the restriction problem for one class of hypersurfaces with vanishing curvature in $\mathbb{R}^n$ with $n$ being odd. We obtain an $L^2-L^p$ restriction estimate, which is optimal except at the endpoint. Furthermore, we establish an $L^s-L^p$ restriction estimate for these hypersurfaces, which is achieved by improving the known $L^{\infty}$ restriction estimate for hypersurfaces with $\frac{n-1}{2}$ positive principal curvatures and $\frac{n-1}{2}$ negative principal curvatures.
\end{abstract}

\maketitle

\begin{center}
\begin{minipage}{120mm}
\end{minipage}
\end{center}




\section{Introduction and main result}

\noindent

Let $B^{n-1}(0,1)$ be the unit ball centered at the origin in $\mathbb{R}^{n-1}$. We define the extension operator
\begin{equation}\label{extensionoperator}
E_Mf(x):= \int_{B^{n-1}(0,1)}f(\xi)e[x_1\xi_1+\cdot\cdot\cdot+x_{n-1}\xi_{n-1}+x_n(M\xi \cdot \xi)]d\xi,
\end{equation}
where $x=(x_1,\cdots,x_n)\in\R^n,\;\xi=(\xi_1,\cdot\cdot\cdot,\xi_{n-1})\in \mathbb{R}^{n-1},\;e(t):= e^{it}$ for $t\in \mathbb{R}$ and
\begin{equation*}
M=\left(
  \begin{array}{cc}
    I_m & O \\
    O & -I_{n-1-m} \\
  \end{array}
\right).
\end{equation*}

E. M. Stein \cite{Stein} proposed the restriction conjecture in the 1960s. Its adjoint form can be stated as follows:
\begin{conjecture}
For any $p>\frac{2n}{n-1}$ and $p\geq \frac{n+1}{n-1}q'$, there holds
\begin{equation}\label{conj:global}
\Vert E_Mf \Vert_{L^p(\mathbb{R}^n)}\leq C_{p,q}\Vert f \Vert_{L^q(B^{n-1}(0,1))}.
\end{equation}
\end{conjecture}

The conjecture in $\mathbb{R}^2$ was proven by Fefferman \cite{F70} and Zygmund \cite{Z74} independently. In $\mathbb{R}^n\;(n\geq 3)$ the conjecture remains open. For the case $m=n-1$, we refer to \cite{ BG, Guth1, Guth2, HR, HZ, Tao03, Wang, WW22} for some partial progress. For the case $1\leq m \leq n-2$, one can see the results in \cite{Barron, BMV20a,GO20,HI, CL17,Lee06, Stova19, Vargas05} and the references therein. In particular, when $n\geq 5$ is odd and $m=\frac{n-1}{2}$, Stein-Tomas theorem is the best known result.

By the standard $\varepsilon$-removal argument in \cite{BG, Tao99}, \eqref{conj:global} can be reduced to a local version as follows:
\begin{conjecture}[Local version on restriction conjecture]\label{conj:local}
Let $p>\frac{2n}{n-1}$ and $p\geq \frac{n+1}{n-1}q'$. Then, for any $\varepsilon>0$, there exists a positive constant $C(\varepsilon)$ such that
\begin{equation}\label{equ:parloc}
  \| E_Mf \|_{L^p(B_R)}\leq C(\varepsilon)R^{\varepsilon}\|f\|_{L^q(B^{n-1}(0,1))},
\end{equation}
where $B_R$ denotes an arbitrary ball with radius $R$ in $\mathbb{R}^n$.
\end{conjecture}

How about the restriction problem for general hypersurfaces with vanishing curvature? Stein \cite{Stein93} first proposed the restriction problem for the hypersurfaces of finite type in $\mathbb{R}^n$. In this direction, I. Ikromov, M. Kempe and D. M\"{u}ller \cite{IKM10, IM11, IM15} obtained the sharp range of Stein-Tomas type restriction estimate for a large class of smooth degenerate surfaces in $\mathbb{R}^3$ including all analytic cases. Recently, Buschenhenke-M\"{u}ller-Vargas \cite{BMV} study the $L^s-L^p$ restriction estimates for certain surfaces of finite type in $\R^3$ via a bilinear method. By developing rescaling techniques associated with such finite type surfaces, the authors of \cite{LMZ} improve the results of \cite{BMV} in the $L^{\infty}-L^p$ setting.

In this article, we consider the hypersurface in $\mathbb{R}^{2k+1}$ given by
\begin{equation*}
\Sigma^{m}:=\Big\{(\xi,\eta,\vert \xi \vert^m-\vert \eta \vert^m):\;(\xi,\eta)\in B^k_1\times B^k_1\Big\},
\end{equation*}
where $m\geq 4$ is an even number, $k\geq 1$ is an integer and $B^k_1$ denotes the unit ball centered at the origin in $\mathbb{R}^k$. For each subset $Q\subset B^k_1\times B^k_1$, we denote the Fourier extension operator associated with $\Sigma^{m}$ by
\begin{equation*}
  \mathcal E^{m}_{Q}g(x):= \int_Q g(\xi,\eta)e\big(x'\cdot \xi+x''\cdot\eta+ x_{2k+1}(\vert \xi \vert^m-\vert \eta \vert^m)\big)d\xi d\eta,
\end{equation*}
where
\[x:=(x',x'',x_{2k+1}),\;x'\in \mathbb{R}^k,\;x'' \in \mathbb{R}^k,\;x_{2k+1}\in\R.\]
Since $m\geq 4$ is an even number, $\vert \xi \vert^m-\vert \eta \vert^m$ is a real homogenous polynomial on $\mathbb{R}^{2k}$ whose Hessian matrix has rank at least $k$ whenever $(\xi,\eta)\neq (0,0)$. Note that $k>\frac{4k}{m}$ when $m>4$. We can deduce decay for the Fourier transform of measures carried on the hypersurface $\Sigma^{m}$ from classical results in \cite{Stein93}
\[\vert \widehat{d\mu}(x)\vert\lesssim (1+\vert x \vert)^{-\frac{2k}{m}},\]
where $m\geq 6$ is an even number. By the classical result by Greenleaf \cite{Gl81} on Stein-Tomas estimates, we derive
\begin{equation}\label{equ:m6}
\|\mathcal E^{m}_{B^k_1\times B^k_1}g\|_{L^{p}(\mathbb{R}^{2k+1})}\lesssim \|g\|_{L^2(B^k_1\times B^k_1)},\quad  p\geq \frac{2k+m}{k},
\end{equation}
where $m \geq 6$ is an even number.

The range of exponent $p$ in \eqref{equ:m6} is sharp. To see it, we construct a Knapp example. Let $K$ be a large number with $1\ll K \ll R^{\varepsilon}$. Taking $g=\chi_{G}$ to be the characteristic function of the set $G$ with
\begin{equation*}
G:=\{(\xi,\eta)\in \mathbb{R}^k\times \mathbb{R}^k:\;\vert \xi \vert \leq R^{-1/m},\;\vert \eta \vert \leq R^{-1/m}\},
\end{equation*}
then, we have
\begin{equation*}
\vert \mathcal E^{m}_{G}g(x) \vert = \Big\vert\int_{G}e\Big(x'\cdot \xi+x''\cdot \eta+x_{2k+1}(\vert \xi \vert^m-\vert \eta \vert^m)\Big)d\xi d\eta \Big\vert \gtrsim R^{-\frac{2k}{m}}
\end{equation*}
provided that $x\in G^{*}$ with
\begin{equation*}
G^{*}:=\{x\in \mathbb{R}^{2k+1}:\;\vert x' \vert \lesssim R^{1/m},\;\vert x'' \vert \lesssim R^{1/m},\;\vert x_{2k+1} \vert \lesssim R\}.
\end{equation*}
Assume that the local version of \eqref{equ:m6}
\begin{equation}\label{equ:goalm}
\|\mathcal E_{B^k_1\times B^k_1}g\|_{L^{p}(B_R)}\leq C(\varepsilon)R^{\varepsilon}\|g\|_{L^2(B^k_1\times B^k_1)}
\end{equation}
holds for certain $p$, we deduce that the following inequality must hold
\begin{equation}
R^{-\frac{2k}{m}}R^{(1+\frac{2k}{m})\frac{1}{p}}\lesssim_{\varepsilon}R^{-\frac{k}{m}+\varepsilon},
\end{equation}
which implies $p\geq \frac{2k+m}{k}$.

Now we focus on the case $m=4$ and abbreviate the corresponding extension operator by
$\mathcal E_{Q}$. Our first result is the following $L^2-L^p$ restriction estimate.

\begin{theorem}\label{thm:main}
Let $p>p_c:=\frac{2k+4}{k}$ and $k\geq 2$. Then, for any $\varepsilon>0$, there exists a positive constant $C(\varepsilon)$ such that
\begin{equation}\label{equ:goalm}
\|\mathcal E_{B^k_1\times B^k_1}g\|_{L^{p}(B_R)}\leq C(\varepsilon)R^{\varepsilon}\|g\|_{L^2(B^k_1\times B^k_1)},
\end{equation}
where $B_R$ denotes an arbitrary ball with radius $R$ in $\mathbb{R}^{2k+1}$.
\end{theorem}
The range of $p$ in Theorem \ref{thm:main} is optimal except at the endpoint.

\begin{remark}\label{geometryRek}
To analyse the extension operator $\mathcal E$, we need to partition the hypersurface $\Sigma$ into small pieces in an appropriate manner. By a direct calculation, we see that the Gaussian curvature of $\Sigma$ vanishes when $\vert \xi \vert=0$ or $\vert \eta \vert=0$. We observe that the hypersurface has nonzero Gaussian curvature if both $\vert \xi \vert$ and $\vert \eta \vert$ are away from zero. In this region, we can adopt Stein-Tomas theorem for hypersurfaces with nonzero Gaussian curvature. Then it reduces to the case $\vert \xi \vert \ll 1$ or $\vert \eta \vert \ll 1$. In other words, we only need to consider small neighborhoods of the submanifolds $\{(\xi,0,\vert \xi \vert^4):\;\xi \in B^k_1\}$ and $\{(0,\eta,\vert \eta \vert^4):\;\eta \in B^k_1\}$ in the hypersurface $\Sigma$. We will adapt the reduction of dimension arguments in \cite{LMZ, LiZheng} to these small neighborhoods.
\end{remark}

Our second result is the following $L^s-L^p$ restriction estimate.

\begin{theorem}\label{mainthm2}
Let $k\geq 2$. Then, for any $\varepsilon>0$, there exists a positive constant $C(\varepsilon)$ such that
\begin{equation}\label{equ:goalthm2}
\|\mathcal E_{B^k_1\times B^k_1}g\|_{L^{p}(B_R)}\leq C(\varepsilon)R^{\varepsilon}\|g\|_{L^{s}(B^k_1\times B^k_1)}
\end{equation}
for $p>\frac{(k+2)(4k^2+6k+1)s}{(k+2)(2k^2+k)s-4k^2-3k}$ and $2< s < \infty$.
\end{theorem}

By Theorem \ref{thm:main} and interpolation, Theorem \ref{mainthm2} follows from the  $L^{\infty}-L^p$ estimate:
\begin{equation}\label{eq:reducedmain2}
\|\mathcal E_{B^k_1\times B^k_1}g\|_{L^{p}(B_R)}\leq C(\varepsilon)R^{\varepsilon}\|g\|_{L^{\infty}(B^k_1\times B^k_1)},
\end{equation}
where $p>p_c-\frac{4k+3}{k(2k+1)}$ and $p_c$ is defined as in Theorem \ref{thm:main}.

\vskip 0.2in

The paper is organized as follows. In Section 2, we give the proof of Theorem \ref{thm:main}. In Section 3, we prove \eqref{eq:reducedmain2}.

\vskip 0.2in

{\bf Notations:} For nonnegative quantities $X$ and $Y$, we will write $X\lesssim Y$ to denote the estimate $X\leq C Y$ for some
large constant $C$ which may vary from line to line and depend on various parameters. If $X\lesssim Y\lesssim X$, we simply write $X\sim Y$. Dependence of implicit constants on the power $p$ or the dimension will be suppressed; dependence on additional parameters will be indicated by subscripts. For example, $X\lesssim_u Y$ indicates $X\leq CY$ for some $C=C(u)$. For any set $E \subset \mathbb{R}^d$, we use $\chi_{E}$ to denote the characteristic function on $E$. Usually, Fourier transform on $\mathbb{R}^d$ is defined by
\begin{equation*}
\widehat{f}(\xi):=(2\pi)^{-d}\int_{\mathbb{R}^d}e^{-ix\cdot \xi}f(x)\;dx.
\end{equation*}

\section{Proof of Theorem \ref{thm:main}}

In this section, we prove Theorem \ref{thm:main}. Let $Q_p(R)$ denote the least number such that
\begin{equation}\label{equ:defqpr}
  \|\mathcal{E}_{\Omega}g\|_{L^p(B_R)}\leq Q_p(R)\|g\|_{L^2(\Omega)},
\end{equation}
for all $g \in L^2(\Omega)$. Here we use $\Omega$ to denote $B^k_1\times B^k_1$.

Let $K=R^{\varepsilon^{100k}}$. We divide $\Omega$ into $\bigcup\limits_{j=0}^3\Omega_j$, as in Figure 1 below.

\begin{center}
 \begin{tikzpicture}[scale=0.6]

\draw[->] (-0.2,0) -- (4,0) node[anchor=north]{};
\draw[->] (0,-0.2) -- (0,4)  node[anchor=east]{};

\draw (-0.15,0) node[anchor=north] {O}
(3.2,0) node[anchor=north] {$1$}
(1,0) node[anchor=north] {$K^{-\frac14}$};

\draw (0,3.2) node[anchor=east] {1}
      (0,1) node[anchor=east] {$K^{-\frac14}$};

\begin{scope}  
    \clip (0,1.5)--(1.5,0)--(-0.15,0);
   \draw (0,0) circle (1);
\end{scope}

\begin{scope}  
    \clip (0,5)--(5,0)--(-0.15,0);
   \draw (0,0) circle (3.2);
\end{scope}

\path (1.5,-1.3) node(caption){$\xi$-space };  


 \draw[->] (5.8,0) -- (10,0) node[anchor=north] {};
\draw[->] (6,-0.2) -- (6,4)  node[anchor=east] {};
 \draw (5.65,0) node[anchor=north] {O}
 (9.2,0) node[anchor=north] {1}
 (7,0) node[anchor=north] {\small $K^{-\frac14}$};

 \draw (6,3.2) node[anchor=east] {1}
      (6,1) node[anchor=east] {${K^{-\frac14}}$};

\begin{scope}  
    \clip (5.85,0)--(10,0)--(7,9);
   \draw (6,0) circle (1);
\end{scope}

\begin{scope}  
    \clip (5.65,0)--(10,0)--(6.5,9);
   \draw (6,0) circle (3.2);
\end{scope}

\path (7.5,-1.3) node(caption){$\eta$-space };  


 \draw[->] (11.8,0) -- (16,0) node[anchor=north] {$|\xi|$};
\draw[->] (12,-0.2) -- (12,4)  node[anchor=east] {$|\eta|$};

  \draw (11.65,0) node[anchor=north] {O}
 (15.2,0) node[anchor=north] {1}
(12.8,0) node[anchor=north] {${K^{-\frac14}}$}
(14,2.6) node[anchor=north] {$\Omega_0$}
(14,0.8) node[anchor=north] {$\Omega_1$}
(12.5,0.8) node[anchor=north] {$\Omega_3$}
(12.5,2.6) node[anchor=north] {$\Omega_2$};

 \draw (12,3.2) node[anchor=east] {1}
 (12,1) node[anchor=east] {${K^{-\frac14}}$};

  \draw[thick] (12,3.2) -- (15.2,3.2)
             (15.2,0) -- (15.2,3.2)
             (12,1) -- (15.2,1)
             (13,0) -- (13,3.2);

\path (5.5,-2.3) node(caption){Figure 1 };  

\end{tikzpicture}
 \end{center}

where
\begin{align*}
&\Omega_0:=A^k\times A^k,\quad
\Omega_1:=A^k\times B^k_{K^{-1/4}},\\
&\Omega_2:=B^k_{K^{-1/4}}\times A^k,\quad
\Omega_3:=B^k_{K^{-1/4}}\times B^k_{K^{-1/4}},\\
&A^k:=B^k_1\setminus B^k_{K^{-1/4}}.
\end{align*}

In this setting, we have
\begin{equation}\label{equ:egr03}
   \|\mathcal{E}_{\Omega}g\|_{L^p(B_R)}\leq \sum_{j=0}^{3}\big\|\mathcal E_{\Omega_j}g \big\|_{L^p(B_R)}.
\end{equation}
Since the hypersurface corresponding to the region $\Omega_0$ possesses nonzero Gaussian curvature with lower bounds depending only on $K$,
we have by Stein-Tomas theorem \cite{Stein86,Tomas}
\begin{equation}\label{equ:ome0est}
  \|\mathcal E_{\Omega_0}g \|_{L^p(B_R)}\lesssim K^{O(1)}\|g\|_{L^2(\Omega_0)},
\end{equation}
for $p>\frac{2k+2}{k}$.

For $\Omega_3$, by the change of variables $\xi=K^{-\frac14}\tilde{\xi}, \eta=K^{-\frac14}\tilde{\eta}$, we have
\begin{align*}
\mathcal E_{\Omega_3}g(x)=& \int_{\Omega_3} g\big(\xi,\eta \big)e\Big(x'\cdot \xi+x''\cdot \eta+x_{2k+1}(\vert \xi \vert^4-\vert \eta \vert^4)\Big)d\xi d\eta\\
=&\int_{B^k_1\times B^k_1} \tilde{g}(\tilde{\xi},\tilde{\eta})e\Big(K^{-\frac{1}{4}}x'\cdot\tilde{\xi}+K^{-\frac{1}{4}}x''\cdot\tilde{\eta}+K^{-1}x_{2k+1}(\vert \tilde{\xi} \vert^4-\vert \tilde{\eta}\vert^4)\Big)d\tilde{\xi}d\tilde{\eta}\\
=& \big(\mathcal E_{B^k_1\times B^k_1}\tilde{g}\big)(\tilde{x}),
\end{align*}
where
\begin{align*}
\tilde{g}(\tilde{\xi},\tilde{\eta}):=&K^{-\frac{k}{2}}g(K^{-\frac{1}{4}}\tilde{\xi},K^{-\frac{1}{4}}\tilde{\eta}),
\end{align*}
and
\begin{equation*}
\tilde{x}=(K^{-\frac{1}{4}}x',K^{-\frac{1}{4}}x'',K^{-1}x_{2k+1}).
\end{equation*}
Therefore, we derive
\begin{align}\nonumber
\|\mathcal E_{\Omega_3}{g}\|_{L^p(B_R)}\leq & K^{\frac{k+2}{2p}}\|\mathcal E_{B^k_1\times B^k_1}\tilde{g}\|_{L^p(B_{\frac{R}{K^{1/4}}})}\\\nonumber
\leq&K^{\frac{k+2}{2p}}Q_p\Big(\tfrac{R}{K^{\frac14}}\Big)\|\tilde{g}\|_{L^2(\Omega)}\\\label{equ:ome3est}
\leq& C(\varepsilon)K^{\frac{k+2}{2p}-\frac{k}4}Q_p\Big(\tfrac{R}{K^{\frac14}}\Big)\|g\|_{L^2(\Omega_3)}.
\end{align}

It suffices to consider the estimate for $\Omega_1$-part and $\Omega_2$-part. By symmetry, we only need to estimate the contribution from $\Omega_1$-part. We decompose $\Omega_1$ into
\[\Omega_1=\bigcup \Omega_{\lambda},\;\Omega_{\lambda}=A^k_{\lambda}\times B^k_{K^{-1/4}},\]
where $\lambda$ is a dyadic number satisfying $K^{-\frac14}\leq \lambda \leq \frac{1}{2}$ and
\[A^k_{\lambda}:=\{\xi \in B^k_1:\;\lambda \leq \vert \xi \vert \leq 2\lambda\},\]
as in Figure 2 below.

\begin{center}
\includegraphics[width=120mm]{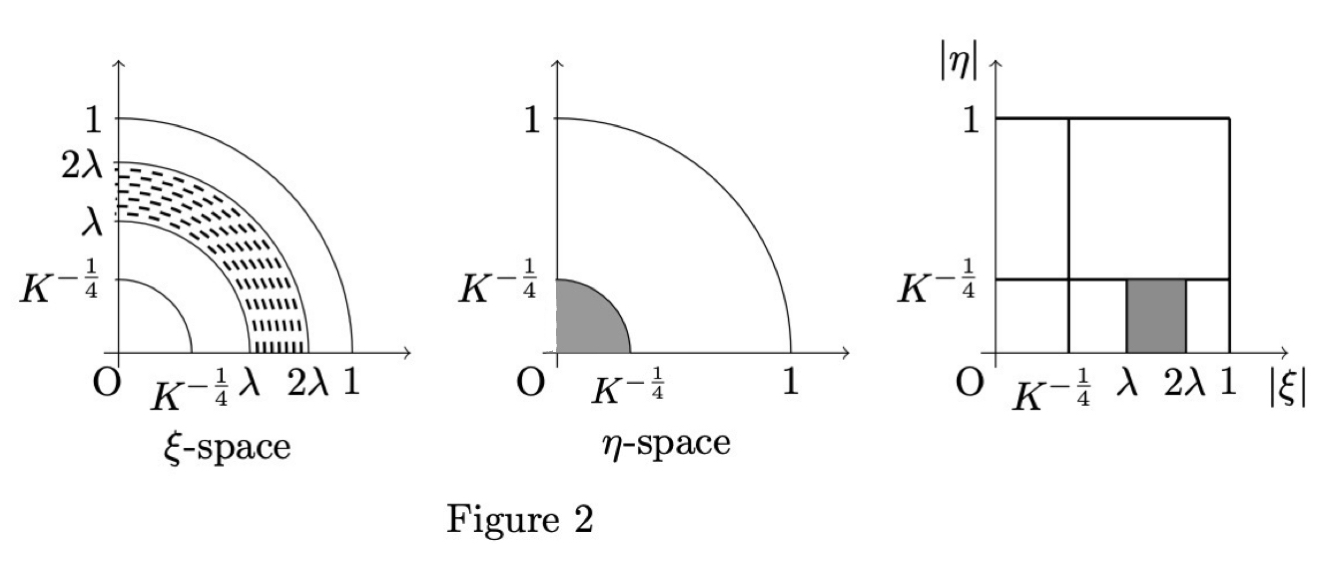}

\end{center}

It suffices to estimate the contribution from each $\Omega_{\lambda}$. We will use an induction on scale argument. For this purpose, we establish a decoupling inequality for
\[\Sigma_{1,\lambda}:=\{(\xi,\vert \xi \vert^4):\;\xi\in\mathbb{R}^k,\;\xi\in A^k_{\lambda}\}\]
in $\mathbb{R}^{k+1}$.

Let $E$ denote the Fourier extension operator associated with $\Sigma_{1,\lambda}$ in $\mathbb{R}^{k+1}$. We cover the region $A^k_{\lambda}$ by $\lambda^{-1}K^{-1/2}$-balls $\tau$. Note that when $\lambda=\frac{1}{2}$, $\Sigma_{1,\frac{1}{2}}$ possesses positive definite second fundamental form in $\mathbb{R}^{k+1}$. Each $\Sigma_{1,\lambda}$ can be transformed into $\Sigma_{1,\frac{1}{2}}$. After rescaling, each $\lambda^{-1}K^{-1/2}$-ball becomes $\lambda^{-2}K^{-1/2}$-ball. Thus, by Bourgain-Demeter's decoupling inequality in \cite{BD15} for perturbed paraboloid, we obtain
\begin{lemma}\label{fracDec}
Let $p_c=\frac{2(k+2)}{k}$ and $0<\delta \ll \varepsilon$. For $p=p_c+\delta$, there holds
\begin{equation}\label{equ:fracDec}
\Vert E_{A^k_{\lambda}}g \Vert_{L^p(B^{k+1}_K)}\lesssim_{\varepsilon} K^{\varepsilon}\Big(\sum_{\tau}\Vert E_{\tau}g \Vert^2_{L^p(\omega_{B^{k+1}_K})}\Big)^{1/2},
\end{equation}
where $\omega_{B^{k+1}_K}(y):=\big(1+\vert \frac{y-c(B^{k+1}_K)}{K} \vert \big)^{-200k}$ denotes the standard weight function adapted to the ball $B^{k+1}_K$. Here $B^{k+1}_K$ represents an arbitrary ball of radius $K$ in $\mathbb{R}^{k+1}$ and $c(B^{k+1}_K)$ denotes its center.
\end{lemma}

With Lemma \ref{fracDec} in hand, we prove the decoupling inequality for the region $\Omega_{\lambda}$ by freezing the $x''$ variable as follows. Fix a bump function $\varphi \in C^{\infty}_c(\mathbb{R}^{2k+1})$ with ${\rm supp}\; \varphi \subset B^{2k+1}(0,1)$ and $\vert \check{\varphi}(x)\vert \geq 1$ for all $x \in B^{2k+1}(0,1)$. Defining $F:=\mathcal{F}^{-1}({\varphi}_{K^{-1}}\cdot \mathcal{E}_{\Omega_{\lambda}}g)$, where $\varphi_{K^{-1}}(\zeta):= K^{2k+1}\varphi(K\zeta),\;\zeta:=(\xi,\eta,s)\in \mathbb{R}^k\times \mathbb{R}^k\times \mathbb{R}=\mathbb{R}^{2k+1}$.
Then we denote $F(\cdot,x'',\cdot)$ by $G$. By the argument in \cite{Guth18}, it is easy to see that ${\rm supp}\;\hat{G}$ is contained in the projection of ${\rm supp}\;\hat{F}$ on the hyperplane $\eta=0$, that is, in the $K^{-1}$-neighborhood of
\[\{(\xi,\vert \xi \vert^4):\xi \in A^k_{\lambda}\}\]
in $\mathbb{R}^{k+1}$. Applying an equivalent form of Lemma \ref{fracDec} to $G$, we get
\[ \|G\|_{L^p(\mathbb{R}^{k+1})}\leq C_{\varepsilon}K^{\varepsilon} \Big(\sum_{\bar{\tau}}\Vert G_{\bar{\tau}}\Vert^2_{L^p(\mathbb{R}^{k+1})}\Big)^{1/2},\]
namely,
\[\|F(\cdot,x_2,\cdot)\|_{L^p(\mathbb{R}^{k+1})}\leq C_{\varepsilon}K^{\varepsilon} \Big(\sum_{\bar{\tau}}\Vert F_{\bar{\tau}}(\cdot,x_2,\cdot)\Vert^2_{L^p(\mathbb{R}^{k+1})}\Big)^{1/2},\]
where
\[G_{\bar{\tau}}:=\mathcal{F}^{-1}(\hat{G}\chi_{\bar{\tau}})\]
and $\bar{\tau}$ denotes the $K^{-1}$-neighborhood of $\tau$ in $\mathbb{R}^{k+1}$. Integrating on both sides of the above inequality with respect to $x_2$-variable in $\mathbb{R}^k$, we derive
\[\|F\|_{L^p(\mathbb{R}^{2k+1})}\leq C_{\varepsilon}K^{\varepsilon} \Big(\sum_{\bar{\tau}}\Vert F_{\bar{\tau}}\Vert^2_{L^p(\mathbb{R}^{2k+1})}\Big)^{1/2}.\]
Thus, we have
\begin{align*}
\|\mathcal{E}_{\Omega_{\lambda}}g \|_{L^p(B_K)}\lesssim \|F\|_{L^p(\mathbb{R}^{2k+1})}
\leq& C_{\varepsilon}K^{\varepsilon}\Big(\sum_{\bar{\tau}}\Vert F_{\bar{\tau}}\Vert^2_{L^p(\mathbb{R}^{2k+1})}\Big)^{1/2}\\
\leq& C_{\varepsilon}K^{\varepsilon}\Big(\sum_{\tau}
\Vert\mathcal{E}_{\tau\times B^k_{K^{-1/4}}}g\Vert^2_{L^p(w_{B_K})}\Big)^{1/2}.
\end{align*}

Summing over all the balls $B_K\subset B_R$ we obtain
\begin{equation}\label{equ:fracDec1R}
\Vert \mathcal{E}_{\Omega_{\lambda}}g \Vert_{L^p(B_R)}\lesssim_{\varepsilon}K^{\varepsilon}\Big(\sum_{\tau}\Vert \mathcal{E}_{\tau \times B^k_{K^{-1/4}}}g \Vert^2_{L^p(\omega_{B_R})}\Big)^{1/2},
\end{equation}
where $\omega_{B_R}$ denotes the weight function adapted to the ball $B_R$. Here $B_R$ represents the ball centered at the origin of radius $R$ in $\mathbb{R}^{2k+1}$.

We apply rescaling to the term $\Vert \mathcal{E}_{\tau \times B^k_{K^{-1/4}}}g \Vert$. Taking the change of variables
\[\xi=\xi^{\tau}+\lambda^{-1}K^{-1/2}\tilde{\xi},\;\eta=K^{-1/4}\tilde{\eta}\]
we rewrite
\begin{equation}\label{eq:rewrite}
\vert \mathcal{E}_{\tau\times B^k_{K^{-1/4}}}g(x)\vert=\Big\vert\int_{\Omega}\tilde{g}(\tilde{\xi},\tilde{\eta})e[\tilde{x}'\cdot \tilde{\xi}+ \tilde{x}''\cdot \tilde{\eta}+\tilde{x}_{2k+1}(\psi_1(\tilde{\xi})-\vert\tilde{\eta}\vert^4)]d\tilde{\xi}d\tilde{\eta}\Big\vert,
\end{equation}
where $\xi^{\tau}$ denotes the center of $\tau$,
\[\tilde{g}(\tilde{\xi},\tilde{\eta}):=\lambda^{-k}K^{-\frac{3k}{4}}g(\xi^{\tau}+\lambda^{-1}K^{-1/2}\tilde{\xi},K^{-1/4}\tilde{\eta}),\]
\[\tilde{x}':=\lambda^{-1}K^{-1/2}x'+(K^{-1}\vert \xi^{\tau} \vert^4+4\lambda^{-1}K^{-3/2}\vert \xi^{\tau} \vert^2
x_{2k+1})\xi^{\tau},\]
\[\tilde{x}'':=K^{-1/4}x'',\;\tilde{x}_{2k+1}:=K^{-1}x_{2k+1}\]
and
\[\psi_1(\tilde{\xi}):=\lambda^{-2}\vert \xi^{\tau} \vert^2\vert \tilde{\xi} \vert^2+4\lambda^{-2}\vert \langle \xi^{\tau}, \tilde{\xi}\rangle\vert^2+4\lambda^{-3}K^{-1/2}\langle \xi^{\tau}, \tilde{\xi}\vert \tilde{\xi} \vert^2\rangle+\lambda^{-4}K^{-1}\vert \tilde{\xi} \vert^4.\]

We claim that the hypersurface
\begin{equation*}
\tilde{\Sigma}_1:=\{(\tilde{\xi},\psi_1(\tilde{\xi})):\;\tilde{\xi}\in B^k_1\}
\end{equation*}
has positive definite second fundamental form in $\mathbb{R}^{k+1}$. It can be verified as follows. For simplicity, we show the calculation only for $k=2$. By a direct computation, the Hessian matrix of the function $\psi_1(\tilde{\xi})$ is
\begin{equation}
\left(
  \begin{array}{cc}
    \partial^2_{11}\psi_1 & \partial^2_{12}\psi_1 \\
    \partial^2_{21}\psi_1 & \partial^2_{22}\psi_1 \\
  \end{array}
\right),
\end{equation}
where
\[\partial^2_{11}\psi_1(\tilde{\xi})=\lambda^{-2}\big(2\vert\xi^{\tau}\vert^2+8(\xi^{\tau}_1)^2\big)
+\lambda^{-3}K^{-1/2}(24\xi^{\tau}_1\tilde{\xi}_1+4\xi^{\tau}_2\tilde{\xi}_2)
+\lambda^{-4}K^{-1}\big(12(\tilde{\xi}_1)^2+4(\tilde{\xi}_2)^2\big),\]
\[\partial^2_{12}\psi_1(\tilde{\xi})=\partial^2_{21}\psi_1(\tilde{\xi})=8\lambda^{-2}\xi^{\tau}_1\xi^{\tau}_2
+8\lambda^{-3}K^{-1/2}(\xi^{\tau}_1\tilde{\xi}_2+\xi^{\tau}_2\tilde{\xi}_1)
+8\lambda^{-4}K^{-1}\tilde{\xi}_1\tilde{\xi}_2,\]
and
\[\partial^2_{22}\psi_1(\tilde{\xi})=\lambda^{-2}\big(2\vert\xi^{\tau}\vert^2+8(\xi^{\tau}_2)^2\big)
+\lambda^{-3}K^{-1/2}(24\xi^{\tau}_2\tilde{\xi}_2+4\xi^{\tau}_1\tilde{\xi}_1)
+\lambda^{-4}K^{-1}\big(12(\tilde{\xi}_2)^2+4(\tilde{\xi}_1)^2\big).\]

Without loss of generality, we can assume $\xi^{\tau}_2=0$. Then one can deduce from the fact $K^{-1/4}\leq \lambda \leq \frac{1}{2}$ that the two eigenvalues of Hessian matrix of $\psi_1$ are $\sim 1$ and
\[\vert\partial^{\alpha}\psi_1\vert\lesssim 1,\;3\leq \vert \alpha \vert\leq 4,\;\;\vert\partial^{\beta}\psi_1\vert=0,\;\vert \beta \vert\geq 5,\]
on $B^k_1$. We say that such a phase function is admissible. This terminology will be used several times later.

To estimate the right-hand side of \eqref{eq:rewrite}, we denote by
\[\tilde{\mathcal{E}}_\Omega f(\tilde{x})
:=\int_{\Omega}f(\tilde{\xi},\tilde{\eta})e[\tilde{x}'\tilde{\xi}+\tilde{x}''\tilde{\eta}+\tilde{x}_{2k+1}(\psi_1(\xi)-\vert \tilde{\eta}\vert^4)]d\tilde{\xi}d\tilde{\eta}.\]

\begin{proposition}\label{newprop}
Let $0<\delta \ll \varepsilon$ and $p=p_c+\delta$. There holds
\begin{equation}\label{newequ}
\Vert \tilde{\mathcal{E}}_{\Omega}f \Vert_{L^p(B_R)}\lesssim_{\varepsilon}R^{\varepsilon}\Vert f \Vert_{L^2(\Omega)}.
\end{equation}
\end{proposition}

Assume that Proposition \ref{newprop} holds for a while, we have by rescaling that
\begin{equation}\label{equ:tauestimate}
\Vert \mathcal{E}_{\tau\times B^k_{K^{-1/4}}}g \Vert_{L^p(B_R)}\lesssim_{\varepsilon}R^{\varepsilon}\Vert g \Vert_{L^2(\tau\times B^k_{K^{-1/4}})}.
\end{equation}

Plugging \eqref{equ:tauestimate} into \eqref{equ:fracDec1R} we get
\begin{equation}\label{equ:frac1restr}
\Vert \mathcal{E}_{\Omega_{\lambda}}g \Vert_{L^p(B_R)}\lesssim_{\varepsilon} R^{\varepsilon}\Vert g \Vert_{L^2(\Omega_{\lambda})},\;\;p>p_c+\delta.
\end{equation}

Combining \eqref{equ:ome0est}, \eqref{equ:ome3est} and \eqref{equ:frac1restr}, we get
\[Q_p(R)\leq K^{O(1)}+2C_{\varepsilon}R^{\varepsilon}+Q_p(\tfrac{R}{K^{1/4}}).\]
Iterating the above inequality $m\approx [\log_{K}R]$ times we derive that
\[Q_p(R)\lesssim_{\varepsilon}R^{\varepsilon}.\]
This completes the proof of Theorem \ref{thm:main}.

Now we turn to prove Proposition \ref{newprop}. Let $A_p(R)$ denote the least number such that
\begin{equation}\label{equ:defapr}
  \|\tilde{\mathcal E}_{\Omega}f\|_{L^p(B_R)}\leq A_p(R)\|f\|_{L^2(\Omega)}
\end{equation}
holds for all $f\in L^2(\Omega)$.

We decompose $\Omega$ into $\tilde{\Omega}_0\bigcup \tilde{\Omega}_1$, as in Figure 3 below.
\begin{center}
 \begin{tikzpicture}[scale=0.7]
\draw[->] (-0.2,0) -- (4,0) node[anchor=north]{};
\draw[->] (0,-0.2) -- (0,4)  node[anchor=east]{};

\draw (-0.15,0) node[anchor=north] {O}
(3.2,0) node[anchor=north] {$1$};

\draw (0,3.2) node[anchor=east] {1};
%

\begin{scope}  
    \clip (0,5)--(5,0)--(-0.15,0);
   \draw (0,0) circle (3.2);
\end{scope}

\path (1.5,-1.3) node(caption){$\xi$-space };  


 \draw[->] (5.8,0) -- (10,0) node[anchor=north] {};
\draw[->] (6,-0.2) -- (6,4)  node[anchor=east] {};
 \draw (5.65,0) node[anchor=north] {O}
 (9.2,0) node[anchor=north] {1}
 (7,0) node[anchor=north] {\small $K^{-\frac14}$};

 \draw (6,3.2) node[anchor=east] {1}
      (6,1) node[anchor=east] {${K^{-\frac14}}$};

\begin{scope}  
    \clip (5.85,0)--(10,0)--(7,9);
   \draw (6,0) circle (1);
\end{scope}

\begin{scope}  
    \clip (5.65,0)--(10,0)--(6.5,9);
   \draw (6,0) circle (3.2);
\end{scope}

\path (7.5,-1.3) node(caption){$\eta$-space };  


 \draw[->] (11.8,0) -- (16,0) node[anchor=north] {$|\xi|$};
\draw[->] (12,-0.2) -- (12,4)  node[anchor=east] {$|\eta|$};

  \draw (11.65,0) node[anchor=north] {O}
 (15.2,0) node[anchor=north] {1}
(13.5,2.6) node[anchor=north] {$\tilde\Omega_0$}
(13.5,0.9) node[anchor=north] {$\tilde\Omega_1$};

 \draw (12,3.2) node[anchor=east] {1}
 (12,1) node[anchor=east] {${K^{-\frac14}}$};

  \draw[thick] (12,3.2) -- (15.2,3.2)
             (15.2,0) -- (15.2,3.2)
             (12,1) -- (15.2,1);

\path (5.5,-2.3) node(caption){Figure 3 };  
\end{tikzpicture}
 \end{center}
where \[\tilde{\Omega}_0:=B^k_1\times A^k,\quad \tilde{\Omega}_1:=B^k_1\times B^k_{K^{-1/4}}.\]

Since the hypersurface corresponding to the region $\tilde{\Omega}_0$ possesses nonzero Gaussian curvature with lower bounds depending only on $K$, the Stein-Tomas theorem \cite{Tomas, Stein86} implies
\begin{equation}\label{equ:newome0est}
  \|\tilde{\mathcal E}_{\tilde{\Omega}_0}f\|_{L^p(B_R)}\leq K^{O(1)}\|f\|_{L^2(\tilde{\Omega}_0)},
\end{equation}
for $p>\frac{2k+2}{k}$.

It remains to estimate the contribution from the $\tilde{\Omega}_1$-part. As in Section 2, we have the following decoupling inequality for the hypersurface corresponding to the region $\tilde{\Omega}_1$.

\begin{equation}\label{equ:newfracDec1R}
\Vert \tilde{\mathcal{E}}_{\tilde{\Omega}_1}f \Vert_{L^p(B_R)}\lesssim_{\varepsilon} K^{\varepsilon}\Big(\sum_{\tilde{\tau}}\Vert \tilde{\mathcal{E}}_{\tilde{\tau} \times B^k_{K^{-1/4}}}f \Vert^2_{L^p(\omega_{B_R})}\Big)^{1/2},
\end{equation}
where $\tilde{\tau}$ denotes $K^{-1/2}$-ball in $\mathbb{R}^{k}$.

Without loss of generality, we may assume that $\tilde{\tau}$ is centered at the origin. Taking the change of variable
\[\tilde{\xi}=K^{-1/2}\bar{\xi},\;\tilde{\eta}=K^{-1/4}\bar{\eta}\]
we rewrite
\[\vert\tilde{\mathcal{E}}_{\tilde{\tau}\times B^k_{K^{-1/4}}}f \vert=\Big\vert \int_{\Omega}\bar{f}(\bar{\xi},\bar{\eta})
e[\bar{x}'\bar{\xi}+\bar{x}''\bar{\eta}+\bar{x}_{2k+1}(\bar{\psi}_1(\bar{\xi})-\bar{\eta}^4)]d\bar{\xi}d\bar{\eta}\Big\vert
=:\vert \bar{\mathcal{E}}_{\Omega}\bar{f}(\bar{x})\vert,\]
where
\[\bar{f}(\bar{\xi},\bar{\eta}):=K^{-\frac{3k}{4}}f(K^{-1/2}\bar{\xi},K^{-1/4}\bar{\eta}),\]
\[\bar{x}':=K^{-1/2}\tilde{x}',\;\bar{x}'':=K^{-1/4}\tilde{x}'',\;\bar{x}_{2k+1}:=K^{-1}\tilde{x}_{2k+1},\]
\[\bar{\psi}_1(\bar{\xi}):=K\psi_1(K^{-\frac{1}{2}}\tilde{\xi})\]
and
$\bar{\mathcal{E}}$ denotes the extension operator associated with the new phase function
\[\bar{\psi}_1(\bar{\xi})-\vert\bar{\eta}\vert^4.\]
We observe that $\bar{\psi}_1$ is
also an admissible phase function in $\mathbb{R}^{k+1}$. Noting that $\vert \bar{x} \vert \leq \frac{R}{K^{1/4}}$ and applying induction on scales to the term $\Vert \bar{\mathcal{E}}_{\Omega}\bar{f} \Vert_{L^p(B_{\frac{R}{K^{1/4}}})}$,
we get
\[\Vert \tilde{\mathcal{E}}_{\tilde{\Omega}_1}f \Vert_{L^p(B_R)}\lesssim_\varepsilon K^{\frac{3k+4}{4p}-\frac{3k}{8}+\varepsilon}A_p(\tfrac{R}{K^{1/4}})\Vert f \Vert_{L^2(\tilde{\Omega}_1)}.\]
Since $p>p_c=\frac{2k+4}{k}$, the above inequality can be rewritten as follows:
\[\Vert \tilde{\mathcal{E}}_{\tilde{\Omega}_1}f \Vert_{L^p(B_R)}\lesssim_\varepsilon A_p(\tfrac{R}{K^{1/4}})\Vert f \Vert_{L^2(\tilde{\Omega}_1)}.\]
This together with \eqref{equ:newome0est} yields
\[A_p(R)\leq K^{O(1)}+C_\varepsilon A_p(\tfrac{R}{K^{1/4}}).\]
Iterating the above inequality $m\approx [\log_{K}R]$ times, we derive
\[A_p(R)\lesssim_{\varepsilon}R^{\varepsilon}\]
as required.

\section{Proof of Theorem \ref{mainthm2}}

In this section, we prove Theorem \ref{mainthm2}. First, we establish an auxiliary proposition as follows:

\begin{proposition}\label{mainprop0}
Let $n\geq 5$ be an odd number and $S$ be a given compact smooth hypersurface with boundary in $\mathbb{R}^n$ with nonzero Gaussian curvature. The inequality
\begin{equation}\label{equ:goalformprop0}
\Vert E_Sf \Vert_{L^p(\mathbb{R}^n)}\leq C_{n,p}\Vert f \Vert_{L^{\infty}(B^{n-1}(0,1))}
\end{equation}
holds for $p>\frac{2(n+1)}{n-1}-\frac{2}{n(n-1)}$, where $E_S$ denotes the extension operator associated with the hypersurface $S$ in $\mathbb{R}^{n}$.
\end{proposition}

\begin{remark}\label{remarkformainprop0}
When $S$ has exactly $\frac{n-1}{2}$ positive principal curvatures, Proposition \ref{mainprop0} breaks the threshold of Stein-Tomas theorem, which is the previous best restriction estimate for the hypersurface $S$.
\end{remark}

By the standard $\varepsilon$-removal argument in \cite{Tao99}, Proposition \ref{mainprop0} reduces to the following local version.
\begin{proposition}\label{localmainprop0}
Let $n\geq 5$ be an odd number and $S$ be a given compact smooth hypersurface with boundary in $\mathbb{R}^n$ with nonzero Gaussian curvature. Suppose that $S$ has exactly $\frac{n-1}{2}$ positive principal curvatures. For any $\varepsilon >0$, there exists a positive constant $C_{\varepsilon}$ such that for any sufficiently large $R$
\[\Vert E_Sf \Vert_{L^p(B_R)}\leq C_{\varepsilon}R^{\varepsilon}\Vert f \Vert_{L^{\infty}(B^{n-1}(0,1))}\]
holds for $p>\frac{2(n+1)}{n-1}-\frac{2}{n(n-1)}$.
\end{proposition}

To prove Proposition \ref{localmainprop0}, we recall the wave packet decomposition at scale $R$ following the description in \cite{Guth2, Wang}.

We decompose the unit ball in $\mathbb{R}^{n-1}$ into finitely overlapping small balls $\theta$ of radius $R^{-1/2}$. These small disks are referred to as $R^{-1/2}$-caps. Let $\psi_{\theta}$ be a smooth partition of unity adapted to $\{\theta\}$, and write $f=\sum_{\theta}\psi_{\theta}f$ and define $f_{\theta}:=\psi_{\theta}f$. We cover $\mathbb{R}^{n-1}$ by finitely overlapping balls of radius about $R^{\frac{1+\delta}{2}}$, centered at vectors $\upsilon\in R^{\frac{1+\delta}{2}}\mathbb{Z}^{n-1}$, where $\delta$ is a small number satisfying $\varepsilon^7< \delta \leq \varepsilon^3$. Let $\eta_{\upsilon}$ be a smooth partition of unity adapted to this cover. We can now decompose
\[f=\sum_{\theta,\upsilon}\big(\eta_{\upsilon}(\psi_{\theta}f)^{\wedge}\big)^{\vee}=\sum_{\theta,\upsilon}\eta_{\upsilon}^{\vee}\ast(\psi_{\theta}f).\]
We choose smooth functions $\tilde{\psi}_{\theta}$ such that $\tilde{\psi}_{\theta}$ is supported on $\theta$ but $\tilde{\psi}_{\theta}=1$ on a $cR^{-1/2}$ neighborhood of the support of $\psi_{\theta}$ for a small constant $c>0$. We define
\[f_{\theta,\upsilon}:=\tilde{\psi}_{\theta}[\eta_{\upsilon}^{\vee}\ast(\psi_{\theta}f)].\]
Since $\eta_{\upsilon}^{\vee}(x)$ is rapidly decaying for $\vert x \vert\gtrsim R^{\frac{1-\delta}{2}}$, we have
\[f=\sum_{(\theta,\upsilon):d(\theta)=R^{-1/2}}f_{\theta,\upsilon}+RapDec(R)\Vert f \Vert_{L^2}.\]
Here the notation $d(\theta)$ denotes the diameter of $\theta$, and $RapDec(R)$ means that the quantity is bounded by $O_N(R^{-N})$ for any large integer $N>0$.

The wave packets $E_Sf_{\theta,\upsilon}$ satisfy two useful properties. The first property is that the functions $f_{\theta,\upsilon}$ are approximately orthogonal. The second property is that on the ball $B^n(0,R)$, the function $E_Sf_{\theta,\upsilon}$ is essentially supported on the tube $T_{\theta,\upsilon}$:
\[T_{\theta,\upsilon}:=\{(x',x_n)\in B^n(0,R), \vert x'+2x_n\omega_{\theta}+\upsilon \vert\leq R^{1/2+\delta}\},\]
where $\omega_{\theta}$ is the center of the cap $\theta$.

If $\mathbb{T}$ is a set of $(\theta,\upsilon)$, we say that $f$ is concentrated on wave packets from $\mathbb{T}$ if $f=\sum_{(\theta,\upsilon)\in\mathbb{T}}f_{\theta,\upsilon}+RapDec(R)\Vert f \Vert_{L^2}$.

Now we turn to prove Proposition \ref{localmainprop0}. In fact, we will establish a result which is stronger than Proposition \ref{localmainprop0}.
\begin{proposition}\label{strongermainprop0}
Let $n\geq 5$ be an odd number and $S$ be a given compact smooth hypersurface with boundary in $\mathbb{R}^n$ with nonzero Gaussian curvature. Suppose that $S$ has exactly $\frac{n-1}{2}$ positive principal curvatures. For any $\varepsilon >0$, there exists a positive constant $C_{\varepsilon}$ such that for any sufficiently large $R$
\[\Vert E_Sf \Vert^p_{L^p(B_R)}\leq C^p_{\varepsilon}R^{p\varepsilon}\Vert f \Vert^{\frac{2n}{n-1}}_{L^{2}(B^{n-1}(0,1))}
\max_{d(\theta)=R^{-1/2}}\Vert f_{\theta} \Vert^{p-\frac{2n}{n-1}}_{L^{2}_{avg}(\theta)}\]
holds for $p>\frac{2(n+1)}{n-1}-\frac{2}{n(n-1)},$
where
\[\Vert f_{\theta} \Vert_{L^{2}_{avg}(\theta)}:=\Big(\frac{1}{\vert \theta \vert}\Vert f_{\theta} \Vert^2_{L^2(\theta)}\Big)^{1/2}.\]
\end{proposition}

\begin{proof}
We apply the polynomial partitioning technique in \cite{Guth1} to the $L^p$-norm of $\vert \chi_{B_R}E_Sf \vert$ directly rather than its $BL^{p}_{2,A}$-norm (see \cite{Guth2} for the definition of the $BL^{p}_{k,A}$-norm). By Theorem 0.6 in \cite{Guth1}, for each degree $d\approx \log R$, one can find a non-zero polynomial $P$ of degree at most $d$ so that the complement of its zero set $Z(P)$ in $B_R$ is a union of $O(d^n)$ disjoint cells $U'_i$: $B_R-Z(P)= \bigcup U'_i,$
and the $L^p$-norm is roughly the same in each cell
\[\Vert Ef \Vert^p_{L^p(U'_i)}\approx d^{-n}\Vert Ef \Vert^p_{L^p(B_R)}.\]
The cells $U'_i$'s might have various shape. For the purpose of induction on scales, we would like to put it inside a smaller ball of radius $\frac{R}{d}$. To do so, it suffices to multiply $P$ by another polynomial $G$ of degree $nd$, and consider the cells cut-off by the zero set of $P\cdot G$. More precisely, let $G_k, k=1,...,n$ be the product of linear equations whose zero set is a union of hyperplanes parallel to $x_k$-axis, of spacing $\frac{R}{d}$ and interesting $B_R$. The degree of $G_k$ is at most $d$. Denote $\prod^n_{k=1}G_k$ by $G$. Let $Q=P\cdot G$ be the new partitioning polynomial, then we have a new decomposition of $B_R$,
\[B_R-Z(Q)=\bigcup O'_i.\]
The zero set $Z(Q)$ decomposes $B_R$ into at most $O(d^n)$ cells $O'_i$ by Milnor-Thom Theorem. A wave packet $E_Sf_{\theta,\upsilon}$ has negligible contribution to a cell $O'_i$ if its essential support $T_{\theta,\upsilon}$ does not intersect $O'_i$. To analyze how $T_{\theta,\upsilon}$ intersects a cell $O'_i$, we need to shrink $O'_i$ further. We define the wall $W$ as the $R^{\frac{1}{2}+\delta}$-neighborhood of $Z(Q)$ in $B_R$ and define the new cells as $O_i=O'_i-W$.

In summary, we decomposed $B_R$ into $B_R= W \cup O_i$ and get the following.
\begin{equation}\label{decomposition}
\Vert E_Sf \Vert^p_{L^p(B_R)}=\Vert E_Sf \Vert^p_{L^p(W)}+\sum_{i}\Vert E_Sf \Vert^p_{L^p(O_i)}
\end{equation}
and
\begin{equation}\label{equalcell}
\Vert E_Sf \Vert^p_{L^p(O_i)}\lesssim d^{-n}\Vert E_Sf \Vert^p_{L^p(B_R)}.
\end{equation}

We are in the cellular case if $\Vert E_Sf \Vert^p_{L^p(B_R)}\lesssim \sum_{i}\Vert E_Sf \Vert^p_{L^p(O_i)}$. We define
\[E_Sf_i=\sum_{T_{\theta,\upsilon}\cap O_i \neq \emptyset}E_Sf_{\theta,\upsilon}.\]
Since the wave packets $E_Sf_{\theta,\upsilon}$ with $T_{\theta,\upsilon}\cap O_i=\emptyset$ have negligible contribution to $\Vert E_Sf \Vert_{L^p(O_i)}$, we have $\Vert E_Sf \Vert_{L^p(O_i)}=\Vert E_Sf_i \Vert_{L^p(O_i)}+RapDec(R)\Vert f \Vert_{L^2}$. Each tube $T_{\theta,\upsilon}$ intersects at most $d+1$ cells $O_i$. It follows that
\[\sum_i \Vert f_i \Vert^2_{L^2}\lesssim d \Vert f \Vert^2_{L^2}.\]
By \eqref{equalcell} and the definition of the cellular case, there are at least $O(d^n)$ cells $O_i$ such that
\begin{equation}\label{icelldecomposition}
\Vert E_Sf \Vert^p_{BL^p(B_R)}\lesssim d^n \Vert E_Sf_i \Vert^p_{BL^p(O_i)}.
\end{equation}
Since there are $O(d^n)$ cells,
\begin{equation}\label{icell2}
\Vert f_i \Vert_{L^2}\lesssim d^{-\frac{n-1}{2}} \Vert f \Vert_{L^2}
\end{equation}
holds for most of the cells. If we are in the cellular case, we derive by induction on scales and \eqref{icell2} that
\[\Vert E_Sf \Vert^p_{L^p(B_R)}\lesssim d^n \Vert E_Sf_{i} \Vert^p_{L^p(O_i)}\]
\[\lesssim C^p_{\varepsilon}(\frac{R}{d})^{\varepsilon}\Vert f_{i} \Vert^{\frac{2n}{n-1}}_{L^2}\max_{d(\tau)=(\frac{R}{d})^{-1/2}}\Vert f_{i,\tau} \Vert^{p-\frac{2n}{n-1}}_{L^{2}_{avg}(\tau)}\]
\[\lesssim C^p_{\varepsilon}R^{p\varepsilon}d^{-p\varepsilon}\Vert f \Vert^{\frac{2n}{n-1}}_{L^2}\max_{d(\theta)=R^{-1/2}}\Vert f_{\theta} \Vert^{p-\frac{2n}{n-1}}_{L^{2}_{avg}(\theta)}.\]
Since $d\approx \log R$, we take $R$ to be sufficiently large so that the induction closes.

If we are not in the cellular case, then $\Vert E_Sf \Vert_{L^p(B_R)}\lesssim \Vert E_Sf \Vert_{L^p(W)}$. We call it the algebraic case because the $L^p$-norm of $E_Sf$ is concentrate on the neighborhood of an algebraic surface. Only the wave packets $E_Sf_{\theta,\upsilon}$ whose essential supports $T_{\theta,\upsilon}$ intersect $W$ contribute to $\Vert E_Sf \Vert_{L^p(W)}$. Depending on how they intersect, we identify a tangential part, which consists of the wave packets tangential to $W$, and a transversal part, which consists of the wave packets intersecting $W$ transversely. In \cite{Guth1}, Guth gives the definition of the tangential tubes and the transversal tubes which we recall here. We cover $W$ with finitely overlapping balls $B_k$ of radius $\rho:=R^{1-\delta}$.

\begin{definition}
$\mathbb{T}_{k,tang}$ is the set of all tubes $T$ obeying the following two condition:\\
(1)\;$T\cap W \cap B_k\neq \emptyset$.\\
(2)\;If $z$ is any non-singular point of $Z(P)$ lying in $2B_k\cap 10T$, then
\[Angle(\upsilon(T),T_zZ)\leq R^{-1/2+2\delta}.\]
We denote $\mathbb{T}_{tang}:=\cup_k \mathbb{T}_{k,tang}$.
\end{definition}

\begin{definition}
$\mathbb{T}_{k,trans}$ is the set of all tubes $T$ obeying the following two condition:\\
(1)\;$T\cap W \cap B_k\neq \emptyset$.\\
(2)\;There exists a non-singular point $z$ of $Z(P)$ lying in $2B_k\cap 10T$ such that
\[Angle(\upsilon(T),T_zZ)> R^{-1/2+2\delta}.\]
We denote $\mathbb{T}_{trans}:=\cup_k \mathbb{T}_{k,trans}$.
\end{definition}
The algebraic part is dominated by
\[\Vert E_Sf \Vert^p_{L^p(W)}\leq \sum_{B_k}\Vert E_Sf_{k,tang} \Vert^p_{L^p(W\cap B_k)}+\sum_{B_k}\Vert E_Sf_{k,trans} \Vert^p_{L^p(W\cap B_k)}.\]

We are in the transversal case if $\Vert E_Sf \Vert^p_{L^p(B_R)}\lesssim \sum_{B_k} \Vert E_Sf_{k,trans} \Vert^p_{L^p(W\cap B_k)}$. The treatment of the transversal case is similar to the cellular case, which requires the following lemma (Lemma 5.7 from \cite{Guth2}) in place of inequality \eqref{icell2}.

\begin{lemma}\label{transgeometriclemma}
Each tube $T$ belongs to at most $Poly(d)$ different sets $\mathbb{T}_{k,trans}$. Here $Poly(d)$ means a quantity bounded by a constant power of $d$.
\end{lemma}
By Lemma \ref{transgeometriclemma}, we have
\begin{equation}\label{transgeoineq}
\sum_{B_k}\Vert f_{k,trans} \Vert^2_{L^2}\lesssim Poly(d)\Vert f \Vert^2_{L^2}.
\end{equation}
If we are in the transversal case, we derive by \eqref{transgeoineq} and induction on scales
\begin{align*}
\Vert E_Sf \Vert^p_{L^p(B_R)}\lesssim& \sum_{B_k}\Vert E_Sf_{k,trans} \Vert^p_{L^p(W\cap B_k)}\\
\lesssim& C^p_{\varepsilon}\rho^{p\varepsilon}\sum_{B_k}\Vert f_{k,trans} \Vert^{\frac{2n}{n-1}}_{L^2}\max_{d(\tau)=\rho^{-1/2}}\Vert f_{k,trans,\tau} \Vert^{p-\frac{2n}{n-1}}_{L^{2}_{avg}(\tau)}\\
\lesssim& C^p_{\varepsilon}\rho^{p\varepsilon}\big(\sum_{B_k}\Vert f_{k,trans} \Vert^2_{L^2}\Big)^{\frac{n}{n-1}}\max_{d(\theta)=R^{-1/2}}\Vert f_{\theta} \Vert^{p-\frac{2n}{n-1}}_{L^{2}_{avg}(\theta)}\\
\lesssim& C^p_{\varepsilon}R^{(1-\delta)p\varepsilon}(Poly(d))^{\frac{n}{n-1}}\Vert f \Vert^{\frac{2n}{n-1}}_{L^2}\max_{d(\theta)=R^{-1/2}}\Vert f_{\theta} \Vert^{p-\frac{2n}{n-1}}_{L^{2}_{avg}(\theta)}.
\end{align*}
Recall that $d\approx logR$. The induction closes for sufficiently large $R$.

Now we turn to discuss the tangential case. We are in the tangential case if $\Vert E_Sf \Vert^p_{L^p(B_R)}\lesssim \sum_{B_k} \Vert E_Sf_{k,tang} \Vert^p_{L^p(W\cap B_k)}$. One has the trivial $L^2$ estimate
\begin{equation}\label{L2energy}
\Vert E_Sf_{k,tang} \Vert_{L^2(W\cap B_k)}\lesssim \rho^{1/2}\Vert f_{k,tang} \Vert_{L^2}.
\end{equation}

The Polynomial Wolff Axioms \cite{KR19} say that $\mathrm{supp}\;f_{k,tang}$ lies in a union of $\lesssim R^{\frac{n-2}{2}+O(\delta)}$ caps $\theta$ of radius $R^{-1/2}$. As a consequence, we have
\begin{equation}\label{keytangentestimate}
\Vert f_{k,tang} \Vert_{L^2}\lesssim R^{-1/4+O(\delta)}\max_{d(\theta)=R^{-1/2}}\Vert f_{\theta} \Vert_{L^{2}_{avg}(\theta)}.
\end{equation}

By interpolating the $L^2$ estimate \eqref{L2energy} with the Stein-Tomas restriction estimate \cite{Tomas, Stein86}, one has
\begin{equation}\label{interpolation}
\Vert E_Sf_{k,tang} \Vert^p_{L^p(W\cap B_k)}\lesssim_{\varepsilon}\rho^{\frac{p}{2}[1-(\frac{1}{2}-\frac{1}{p})(n+1)]}\Vert f_{k,tang} \Vert^p_{L^2}
\end{equation}
for $2\leq p \leq \frac{2(n+1)}{n-1}$.

Applying \eqref{keytangentestimate} to the right-hand side of inequality \eqref{interpolation} we get
\begin{equation}\label{improvedtangBk}
\Vert E_Sf_{k,tang} \Vert^p_{L^p(W\cap B_k)}
\lesssim_{\varepsilon}R^{\frac{n^2+n-1}{2(n-1)}-\frac{pn}{4}+O(\delta)}\Vert f_{k,tang} \Vert^{\frac{2n}{n-1}}_{L^2}\max_{d(\theta)=R^{-1/2}}\Vert f_{\theta} \Vert^{p-\frac{2n}{n-1}}_{L^{2}_{avg}(\theta)}.
\end{equation}
Note that $\Vert f_{k,tang} \Vert_{L^2}\lesssim \Vert f \Vert_{L^2}$. This together with \eqref{improvedtangBk} yields the desired bound for the tangential case whenever $p>\frac{2(n+1)}{n-1}-\frac{2}{n(n-1)}$. Combining the estimates in the cellular case, the transversal case and the tangential case, we conclude that Proposition \ref{strongermainprop0} holds.
\end{proof}

Now we use Proposition \ref{mainprop0} to prove Theorem \ref{mainthm2}. Let $\mathcal{M}_p(R)$ denote the least number such that
\begin{equation}\label{equ:defqpr3}
  \|\mathcal{E}_{\Omega}g\|_{L^p(B_R)}\leq \mathcal{M}_p(R)\|g\|_{L^{\infty}(\Omega)},
\end{equation}
for all $g \in L^{\infty}(\Omega)$, where we use $\Omega$ to denote $B^k_1\times B^k_1$ as before.

Let $K=R^{\varepsilon^{100k}}$. We divide $\Omega$ into $\bigcup\limits_{j=0}^3\Omega_j$, where
\begin{align*}
&\Omega_0:=A^k\times A^k,\quad
\Omega_1:=A^k\times B^k_{K^{-1/4}},\\
&\Omega_2:=B^k_{K^{-1/4}}\times A^k,\quad
\Omega_3:=B^k_{K^{-1/4}}\times B^k_{K^{-1/4}},\\
&A^k:=B^k_1\setminus B^k_{K^{-1/4}}.
\end{align*}

In this setting, we have
\begin{equation}\label{equ:egr03}
   \|\mathcal{E}_{\Omega}g\|_{L^p(B_R)}\leq \sum_{j=0}^{3}\big\|\mathcal E_{\Omega_j}g\big\|_{L^p(B_R)}.
\end{equation}
Since the hyper-surface corresponding to the region $\Omega_0$ possesses nonzero Gaussian curvature with lower bounds depending only on $K$,
we have by Proposition \ref{mainprop0}
\begin{equation}\label{equ:ome0est3}
  \|\mathcal E_{\Omega_0}g\|_{L^p(B_R)}\lesssim K^{O(1)}\|g\|_{L^{\infty}(\Omega_0)},
\end{equation}
for $p>\frac{2k+2}{k}-\frac{1}{k(2k+1)}$.

For $\Omega_3$, we have by rescaling
\begin{equation}\label{equ:ome3est3}
  \|\mathcal E_{\Omega_3}g\|_{L^p(B_R)}\leq CK^{\frac{k+2}{2p}-\frac{k}{2}}Q_p\Big(\tfrac{R}{K^{\frac14}}\Big)\|g\|_{L^{\infty}(\Omega_3)}.
\end{equation}

For $\Omega_1$ and $\Omega_2$, it suffices to consider the estimate for $\Omega_1$-part by symmetry. We decompose $\Omega_1$ into
\[\Omega_1=\bigcup \Omega_{\lambda},\;\Omega_{\lambda}=A^k_{\lambda}\times B^k_{K^{-1/4}},\]
for dyadic $\lambda$ satisfying $K^{-\frac14}\leq \lambda \leq \frac{1}{2}$.

It suffices to estimate the contribution from each $\Omega_{\lambda}$. We cover the region $A^k_{\lambda}$ by $\lambda^{-1}K^{-1/2}$-balls $\tau$. Recall that in Section 2 we have proved the following decoupling inequality:
\begin{equation}\label{equ:fracDec1R3}
\Vert \mathcal{E}_{\Omega_{\lambda}}g \Vert_{L^p(B_R)}\lesssim_{\varepsilon}K^{\varepsilon}\Big(\sum_{\tau}\Vert \mathcal{E}_{\tau \times B^k_{K^{-1/4}}}g \Vert^2_{L^p(\omega_{B_R})}\Big)^{1/2},
\end{equation}
where $\omega_{B_R}$ denotes the weight function adapted to the ball $B_R$. Here $B_R$ represents the ball centered at the origin of radius $R$ in $\mathbb{R}^{2k+1}$.

We apply rescaling to the term $\Vert \mathcal{E}_{\tau \times B^k_{K^{-1/4}}}g \Vert$. Taking the change of variables
\[\xi=\xi^{\tau}+\lambda^{-1}K^{-1/2}\tilde{\xi},\eta=K^{-1/4}\tilde{\eta}\]
we have
\[\vert \mathcal{E}_{\tau\times B^k_{K^{-1/4}}}g(x)\vert=\vert\int_{\Omega}\tilde{g}(\tilde{\xi},\tilde{\eta})e[\tilde{x}'\cdot \tilde{\xi}+ \tilde{x}''\cdot \tilde{\eta}+\tilde{x}_{2k+1}(\psi_1(\tilde{\xi})-\vert\tilde{\eta}\vert^4)]d\tilde{\xi}d\tilde{\eta}\vert,\]
where $\xi^{\tau}$ denotes the center of $\tau$,
\[\tilde{g}(\tilde{\xi},\tilde{\eta}):=\lambda^{-k}K^{-\frac{3k}{4}}g(\xi^{\tau}+\lambda^{-1}K^{-1/2}\tilde{\xi},K^{-1/4}\tilde{\eta}),\]
\[\tilde{x}':=\lambda^{-1}K^{-1/2}x'+(K^{-1}\vert \xi^{\tau} \vert^4+4\lambda^{-1}K^{-3/2}\vert \xi^{\tau} \vert^2
x_{2k+1})\xi^{\tau},\]
\[\tilde{x}'':=K^{-1/4}x'',\;\tilde{x}_{2k+1}:=K^{-1}x_{2k+1}\]
and
\[\psi_1(\tilde{\xi}):=\lambda^{-2}\vert \xi^{\tau} \vert^2\vert \tilde{\xi} \vert^2+4\lambda^{-2}\vert \langle \xi^{\tau}, \tilde{\xi}\rangle\vert^2+4\lambda^{-3}K^{-1/2}\langle \xi^{\tau}, \tilde{\xi}\vert \tilde{\xi} \vert^2\rangle+\lambda^{-4}K^{-1}\vert \tilde{\xi} \vert^4.\]
We know that
the phase function $\psi_1$ is admissible in $\mathbb{R}^{k+1}$.

We denote by
\[\tilde{\mathcal{E}}_\Omega f(\tilde{x})
:=\int_{\Omega}f(\tilde{\xi},\tilde{\eta})e[\tilde{x}'\tilde{\xi}+\tilde{x}''\tilde{\eta}+\tilde{x}_{2k+1}(\psi_1(\xi)-\vert \tilde{\eta}\vert^4)]d\tilde{\xi}d\tilde{\eta}\]
as in Section 2.

\begin{proposition}\label{newprop3}
Let $p>\frac{2k+2}{k}-\frac{1}{k(2k+1)}$. There holds
\begin{equation}\label{newequ3}
\Vert \tilde{\mathcal{E}}_{\Omega}f \Vert_{L^p(B_R)}\lesssim_{\varepsilon}R^{\varepsilon}\Vert f \Vert_{L^{\infty}(\Omega)}.
\end{equation}
\end{proposition}

Assume that Proposition \ref{newprop3} holds for a while, by rescaling and \eqref{equ:fracDec1R3} we have
\begin{equation}\label{equ:frac1restr3}
\Vert \mathcal{E}_{\Omega_{\lambda}}g \Vert_{L^p(B_R)}\lesssim_{\varepsilon} R^{\varepsilon}\Vert g \Vert_{L^\infty(\Omega_{\lambda})}
\end{equation}
holds for $p>\frac{2k+2}{k}-\frac{1}{k(2k+1)}$.

Combining \eqref{equ:ome0est3}, \eqref{equ:ome3est3} and \eqref{equ:frac1restr3} we get
\[\mathcal{M}_p(R)\leq K^{O(1)}+2C_\varepsilon R^{\varepsilon}+\mathcal{M}_p(\tfrac{R}{K^{1/4}}).\]
Iterating the above inequality $m\approx [\log_{K}R]$ times we derive that
\[\mathcal{M}_p(R)\lesssim_{\varepsilon}R^{\varepsilon}.\]
This completes the proof of Theorem \ref{mainthm2}.

{\bf Proof of Proposition \ref{newprop3}:}
Now we turn to prove Proposition \ref{newprop3}. Let $\mathcal{C}_p(R)$ denote the least number such that
\begin{equation}\label{equ:defapr3}
  \|\tilde{\mathcal E}_{\Omega}f\|_{L^p(B_R)}\leq \mathcal{C}_p(R)\|f\|_{L^{\infty}(\Omega)}
\end{equation}
holds for all $f\in L^{\infty}(\Omega)$.

We decompose $\Omega$ into $\tilde{\Omega}_0\bigcup \tilde{\Omega}_1$, where \[\tilde{\Omega}_0:=B^k_1\times A^k,\quad \tilde{\Omega}_1:=B^k_1\times B^k_{K^{-1/4}}\]
as in the proof of Proposition \ref{newprop}.

Since the hypersurface corresponding to the region $\tilde{\Omega}_0$ possesses nonzero Gaussian curvature with lower bounds depending only on $K$, Proposition \ref{mainprop0} implies
\begin{equation}\label{equ:newome0est3}
  \|\tilde{\mathcal E}_{\tilde{\Omega}_0}f\|_{L^p(B_R)}\leq K^{O(1)}\|f\|_{L^{\infty}(\tilde{\Omega}_0)}
\end{equation}
for $p>\frac{2k+2}{k}-\frac{1}{k(2k+1)}$.

It remains to estimate the contribution from the $\tilde{\Omega}_1$-part. We employ inequality \eqref{equ:newfracDec1R}.
\[\Vert \tilde{\mathcal{E}}_{\tilde{\Omega}_1}f \Vert_{L^p(B_R)}\lesssim_{\varepsilon} K^{\varepsilon}\Big(\sum_{\tilde{\tau}}\Vert \tilde{\mathcal{E}}_{\tilde{\tau} \times B^k_{K^{-1/4}}}f \Vert^2_{L^p(\omega_{B_R})}\Big)^{1/2},\]
where $\tilde{\tau}$ denotes $K^{-1/2}$-ball in $\mathbb{R}^{k}$.

Now we apply rescaling to the term
\[\Vert\tilde{\mathcal E}_{\tilde{\tau}\times B^k_{K^{-1/4}}}f \Vert_{L^p(B_R)}.\]
By the similar calculation as in the proof of Theorem \ref{thm:main},
\[\Vert\tilde{\mathcal E}_{\tilde{\tau}\times B^k_{K^{-1/4}}}f \Vert_{L^p(B_R)}\]
reduces to
\[\Vert \bar{\mathcal{E}}_{\Omega}\bar{f} \Vert_{L^p(B_{\frac{R}{K^{1/4}}})},\]
where $\bar{\mathcal{E}}$ denotes the extension operator associated with the new phase function
\[\varphi_1(\zeta)-\vert \omega \vert^4,\;(\zeta,\omega)\in \Omega.\]
We know that $\varphi_1$ is also admissible
in $\mathbb{R}^{k+1}$. So we can apply induction on scales to
\[\Vert \bar{\mathcal{E}}_{\Omega}\bar{f} \Vert_{L^p(B_{\frac{R}{K^{1/4}}})},\]
and get
\[\Vert \tilde{\mathcal{E}}_{\tilde{\Omega}_1}f \Vert_{L^p(B_R)}\leq \mathcal{C}_p(\tfrac{R}{K^{1/4}})\Vert f \Vert_{L^{\infty}(\tilde{\Omega}_1)}.\]
This together with \eqref{equ:newome0est3} yields
\[\mathcal{C}_p(R)\leq K^{O(1)}+C_\varepsilon\mathcal{C}_p(\tfrac{R}{K^{1/4}}).\]
Iterating the above inequality $m\approx [\log_{K}R]$ times we get
\[\mathcal{C}_p(R)\lesssim_{\varepsilon}R^{\varepsilon}\]
as desired.

\vskip 0.2cm

Finally, we give a remark on the case $k=1$.
\begin{remark}\label{rek:kequals1}
Using the same argument as in Section 2 and applying the restriction estimate of Guo-Oh in \cite{GO20} to the $\Omega_0$ case, one can deduce that the inequality
\begin{equation*}
\|\mathcal E_{B^k_1\times B^k_1}g\|_{L^{p}(B_R)}\leq C(\varepsilon)R^{\varepsilon}\|g\|_{L^{\infty}(B^k_1\times B^k_1)}
\end{equation*}
holds for all $p>3.5$ when $k=1$.
\end{remark}

\vskip 0.12in


\subsection*{Declarations} :

$\bullet$ Funding:  This work is supported by the National Key Research and Development Program
of China (No.2021YFA1002500). J. Zheng was supported by NSFC
Grant 12271051 and Beijing Natural Science Foundation 1222019.

$\bullet$  Conflict of interest: There is no conflict of interest.

\begin{center}

\end{center}

\end{document}